\documentclass[11pt]{amsart}
\usepackage{amsfonts}
\usepackage{amssymb}
\usepackage{amsmath}
\usepackage{graphicx}
\usepackage{enumitem}

\newtheorem{thm}{Theorem}[section]
\newtheorem{cor}[thm]{Corollary}

\newtheorem{lem}[thm]{Lemma}
\newtheorem{prop}[thm]{Proposition}
\theoremstyle{definition}
\newtheorem{defn}[thm]{Definition}
\theoremstyle{remark}

\theoremstyle{example}
\newtheorem{exam}[thm]{Example}
\numberwithin{equation}{section}

\newcommand{\Lemma[1]}{\noindent{\textbf{Lemma}}{\bf\textrm{~~#1}}\\}
\newcommand{\Corollary[1]}{\noindent{\textbf{Corollary}}{\bf\textrm{~~#1}}\\}
\newcommand{\Proposition[1]}{\noindent{\textbf{Proposition}}{\bf\textrm{~~#1}}\\}

\begin{document}
\title[Global behavior]{Nonautonomous Riccati difference
equation with real $k$-periodic ($k\geq 1$) coefficients }
\author{ Raouf Azizi}
\address{Raouf Azizi, University of Carthage, Faculty of Sciences of
Bizerte, Department of Mathematics, \thinspace\ 7021, Jarzouna. Tunisia.}
\email{araoufazizi@gmail.com}
\keywords{First order Riccati difference equation,  non-autonomous Riccati difference equation, forbidden
set, periodic solutions, asymptotic stability, dense solutions.}

\begin{abstract}
We study the non-autonomous Riccati difference equation
\[x_{n+1}=\frac{a_nx_n+b_n}{c_nx_n+d_n}, \ n=0,1,2,\cdots\]
where $(a_n)_{n\geq0}, \ (b_n)_{n\geq0}, \ (c_n)_{n\geq0}, \ \text{and} \ (d_n)_{n\geq0}$ are $k$-periodic sequences, $k\geq 1$, with initial value $x_0 \in  \mathbb{R}$.
Precisely we give a detailed analysis of the forbidden set and the character of the solutions.
\end{abstract}

\maketitle

\section{Introduction}

\vspace{-3mm}
Consider the non-autonomous Riccati difference equation
\begin{equation}
x_{n+1}=\frac{a_nx_n+b_n}{c_nx_n+d_n}, \ n=0,1,2,\cdots \label{pre}
\end{equation}
where $(a_n)_{n\geq0}, \ (b_n)_{n\geq0}, \ (c_n)_{n\geq0}, \ (d_n)_{n\geq0}$ are $k$-periodic, $k\geq 1$, and with initial value $x_0 \in  \mathbb{R}$. We assume that
$c_n\neq0 \ \text{and} \ a_nd_n-c_nb_n\neq0, \ \forall \ n\geq0$.\\
\noindent

When $k=1$ the Eq.(\ref{pre}) is reduced to the first
order autonomous Riccati rational equation which has been studied thoroughly (see, e.g.,  \cite{5,6}).
The case $k=2$ is investigated by E.A. Grove, Y. Kostrov, G. Ladas (\cite{1}).
They determined, firstly, the forbidden set of the Eq.(\ref{pre}) or also called the domain of undefinable solutions
(see Definition 2.1), and secondly, the behavior of the solutions of Eq.(\ref{pre}).
In this present paper, we extend the study of the Eq.(\ref{pre}) for any integer $k\geq1$
by determining its forbidden set and the
asymptotic behavior of its solutions. This also improves results in \cite{1, 8}.
For recent progress on the qualitative property on non-autonomous difference equations, see (e.g. , \cite{2,9,8,10,1,5,11}).
\vspace{0.1mm}\\
\noindent

This paper is organized as follows: In Section 2, we recall some definitions and we give some preliminary results
on non-autonomous Riccati difference equation. In Section 3, we study the existence of solutions of
Eq.(\ref{pre}) and their asymptotic behavior. Section 4 is devoted to detailed study of Eq.(\ref{pre}) when $b_n=0$ for
any integer $n$. In Section 5, we study the special case $b_n=0, a_n=d_n=1$, for all $n\geq0$, with $(c_n)_{n\geq0}$
not necessarily periodic. In the Appendix, we give a complete study of Eq.(\ref{pre}) when $k=1$, (using the matrix approach adopted in subsection 2.2).

\section{Preliminaries}

\subsection{ Basic concepts }

For the sake of self-containment and convenience, we recall the following definitions.\\
A \textit{first order non-autonomous difference equation} is an equation of the form
$$ (E) \ x_{n+1}=f_n(x_n), \ n=0,1,\ldots$$
 where $\Omega$ is a subset of $\mathbb{R}$ (usually $\Omega$ is an interval or a union of intervals of $\mathbb{R}$),
 and for any integer $n$,  $f_n : \Omega\rightarrow \mathbb{R}$ is a continuous function.\\
If $f_n=f, \ \forall n \in \mathbb{N}$, the Eq.($E$) is reduced to \textit{first order autonomous difference equation}:
$$x_{n+1}=f(x_n), \ n=0,1,\ldots$$

 A \textit{solution} of Eq.($E$) is a sequence $(x_n)_{n\geq0}$ such that for every integer $n$,
 $x_n \in \Omega$ and $x_{n+1}=f_n(x_n)$.

\begin{defn}
The set of initial values $x_0 \in  \mathbb{R}$ through which the Eq.($E$) is not well defined for all $n\geq0$ is called the
\textit{forbidden set} of this equation, usually denoted by $\mathcal{F}$.
Hence the solution $(x_n)_{n\geq0}$ of Eq.($E$) exists if and only if $x_0 \not\in \mathcal{F}$.
\end{defn}
\begin{defn}
A solution $(x_n)_{n\geq0}$ of Eq.($E$) is called:
\begin{description}
\item[(i)] \textit{periodic} with period $p$ if
$$x_{n+p}=x_n, \ \text{for every} \ n\geq0 \ (*)$$
\item[(ii)] periodic with \textit{prime period} $p$ if it is periodic with period $p$, where $p$ is the least integer
for which $(*)$ holds. In this case, the a \textit{p-tuple} $\{x_n, \ x_{n+1},\ldots,x_{n+p-1}\}$ is called a \textit{p-cycle} of Eq.($E$).
\end{description}
\end{defn}
\begin{defn}
A point $\bar{x}$ is an \textit{equilibrium point} of Eq.($E$) if $f_n(\bar{x})=\bar{x}, \ \text{for all} \ n\geq0$ that is, $x_n=\bar{x}$ for all $n\geq0$ is a constant solution of Eq.($E$).\\
The equilibrium point is said to be:
\begin{description}
\item[(i)] \textit{locally stable} if for every $\epsilon$, there exists $\alpha>0$ such that if $(x_n)_{n\geq0}$ is a
solution of Eq.($E$) with initial value $x_0 \in \Omega\backslash\mathcal{F}$ such that $\mid x_0-\bar{x}\mid<\alpha$ we have
$ \mid x_n-\bar{x}\mid<\epsilon$ for all $n\geq0$.
\item[(ii)] a \textit{global attractor} relative to the set $\mathcal{B}\subset\Omega\backslash\mathcal{F}$
(basin of attraction) if for all solution $(x_n)_{n\geq0}$ of Eq.($E$) with initial value $x_0 \in  \mathcal{B}$,
we have $\displaystyle{\lim_{n\rightarrow+\infty}x_n=\bar{x}}$.
\item[(iii)] a \textit{globally asymptotically stable} relative to the set $\mathcal{B}$ if it is locally stable
and global attractor relative to the set $\mathcal{B}$.
\item[(iv)] \textit{unstable} if it is not stable.
\end{description}
\end{defn}

\subsection{Some results on non-autonomous Riccati difference equation}
\noindent
In this subsection, we consider the Eq.(\ref{pre}) where the coefficients $a_n, \ b_n, \ c_n, \ d_n$ are not
necessarily periodic.
Set
$$A_n=\left(\begin{array}{cc}
a_n & b_n \\
c_n & d_n \\
\end{array}\right) \ \ \text{and} \ \ f_{A_n}(x)=\displaystyle{\frac{a_nx+b_n}{c_nx+d_n}}$$
Then the solution $(x_n)_{n\geq0}$ of Eq.(\ref{pre}) can be expressed as follow:
$$x_n=f_{A_{n-1}}f_{A_{n-2}}\cdots f_{A_0}(x_0), \ n=1,2,\ldots$$
Here $f_{A_{n-1}}f_{A_{n-2}}\cdots f_{A_0}:=f_{A_{n-1}}\circ f_{A_{n-2}}\circ\cdots \circ f_{A_0}$.
\medskip
\begin{prop}\rm{
\noindent
\begin{enumerate}[leftmargin=*]
\item For every invertible matrices $A,B$ we have: $f_A=f_B$ if and only if there exists $\alpha\in\mathbb{R}$ such that $A=\alpha B$.
\item For every $A,B\in\mathcal{M}_2(\mathbb{R})$, $x\in\mathbb{R}$ such that $f_Af_B(x)$ exists we have: $f_{AB}(x)$ exists, and $f_Af_B(x)=f_{AB}(x)$.
\item For every $n\geq0$, $A_0,A_1,\ldots A_n\in\mathcal{M}_2(\mathbb{R})$ and $x\in\mathbb{R}$ such that\\ $f_{A_n}f_{A_{n-1}}\cdots f_{A_0}(x_0)$
  exists, we have: $f_{A_nA_{n-1}\cdots A_0}(x)$ exits, and\\ $f_{A_n}f_{A_{n-1}}\cdots f_{A_0}(x_0)=f_{A_nA_{n-1}\cdots A_0}(x)$. Particularly, if $A_n=A$, for all $n \in \mathbb{N}$, then $(f_A)^n(x)=f_{A^n}(x)$.
\end{enumerate}}
\end{prop}
\medskip

Denote by $$\bar{A}_n=\prod\limits_{i=0}^{n-1}A_{n-1-i}, \ n=0,1,\ldots$$ with $\bar{A}_0:=A_{-1}A_0=I$.\\
Write
$$\bar{A}_n=\left(
                   \begin{array}{cc}
                     \bar{a}_n & \bar{b}_n\\
                     \bar{c}_n & \bar{d}_n \\
                   \end{array}
                 \right)\hspace{2cm}$$
\vspace{1mm}

Using Proposition 2.4, we can give the general form of solutions of Eq.(\ref{pre}) and thus the associated forbidden set.
Then we have the following corollary:
\medskip
\begin{cor}\rm{
\noindent
\begin{enumerate}[leftmargin=*]
\item The forbidden set of Eq.(\ref{re}) is given by:
$$\mathcal{F}=\left\{-\frac{\bar{d}_n}{\bar{c}_n}; \ n\in\mathbb{N}:\  \bar{c}_n\neq0\right\}$$
\item The solution of Eq.(\ref{pre}) with initial value $x \in \mathbb{R}\backslash\mathcal{F}$ is given by:
$$x_n=f_{\bar{A}_n}(x)=\frac{\bar{a}_nx+\bar{b}_n}{\bar{c}_nx+\bar{d}_n}, \ n=0,1,\ldots$$
\end{enumerate}
}\end{cor}
\medskip
In the sequel, we use the convention of notations: $$\prod\limits_{i=p}^q\xi_i=1, \ \sum\limits_{i=p}^q\xi_i=0, \ \text{if} \ p>q$$

There are two specific cases that we can calculate the coefficients $\bar{a}_n, \ \bar{b}_n, \ \bar{c}_n, \ \bar{d}_n$.
The case $A_n=A, \ n\geq0$, corresponds to $\bar{A}_n=A^n$ (see Lemma A.2), and
the case ($b_n=0 \ \text{and} \ d_n=1, \ n\geq0$) (see proposition below).

\medskip
\begin{prop}\rm{
Assume that $b_n=0, \ d_n=1, \ n\geq0$. Then the coefficients $\bar{a}_n, \ \bar{b}_n, \ \bar{c}_n, \ \bar{d}_n$ are given by:
$$\bar{a}_n=\prod_{i=0}^{n-1} a_i, \ \ \bar{b}_n=0, \ \ \bar{c}_n=\sum_{i=0}^{n-1}c_i\prod_{j=0}^{i-1}a_j, \ \ \bar{d}_n=1$$
Particularly, if $a_n=1, \ n\geq0$, then:
$$\bar{a}_n=\bar{d}_n=1, \ \ \bar{c}_n=\sum_{i=0}^{n-1}c_i, \ \ \bar{b}_n=0$$
}\end{prop}
\medskip
Using Proposition 2.6, we deduce the result of S. Stevi\'c on forbidden set (\cite{7}).
\medskip
\begin{cor}\rm{(\cite{7}, Proposition 1)
Assume that $b_n=0, \ d_n=1, \ n\geq0$. Then the forbidden set of Eq.(\ref{pre}) is given by:
$$\mathcal{F}=\left\{-\left(\sum_{i=0}^{n-1}c_i\prod_{j=0}^{i-1}a_j\right)^{-1}; \ n\geq1: \  \sum_{i=0}^{n-1}c_i\prod_{j=0}^{i-1}a_j\neq0\right\}$$
}\end{cor}
\medskip

\medskip

\section{\bf The non-autonomous Riccati difference equation with periodic coefficients}
\medskip
In this section, we determine the forbidden set and the asymptotic behavior of solutions of Eq.(\ref{pre}) and their asymptotic behavior when the coefficients $(a_n)_{n\geq0}, \ (b_n)_{n\geq0}, \ (c_n)_{n\geq0}, \ (d_n)_{n\geq0}$ are k-periodic, $k\geq2$.

\subsection{\bf The forbidden set} \ \vspace{-3mm}\\

Unlike the case $k = 1$, the calculation of the forbidden set is more complicated if $k\geq2$. However, using the matrix approach adopted in subsection 2.2, we can explicitly determine the forbidden set for $k\geq2$.\\
\vspace{-3mm}
\\
Let
$A_n=\left(\begin{array}{cc}
          a_n & b_n \\
          c_n & d_n \\
        \end{array}\right)$ \ be a sequence of matrices associated to the Eq.(\ref{pre}).\\
Set $$B_i=A_{i-1}\cdots A_0A_{k-1}\cdots A_{i}, \ \ 0\leq i\leq k-1$$
\vspace{-3mm}
\\
For any invertible matrix $A$, we denote by $\mathcal{F}_A$ the forbidden set associated to the equation: $x_{n+1}=f_A(x_n)$.
The following result improves and generalizes those of \cite{1}.\vspace{-3mm}\\
\begin{thm}\rm{
The forbidden set of Eq.(\ref{pre}) is given by:
$$\mathcal{F}=\left\{\prod_{j=0}^{i-1}f_{A_j}^{-1}\left(-\frac{d_i}{c_i}\right), \  0\leq i\leq k-2\right\}\bigcup\bigcup\limits_{0\leq i\leq k-1}\prod_{j=0}^{i-1}f_{A_j}^{-1}\left(\mathcal{F}_{B_i}\right)$$
}\end{thm}\vspace{-3mm}
\medskip
\begin{proof}
It is clear that if $x_0\in \mathcal{F}$ the solutions $(x_n)_{n\geq0}$ doesn't  exist.\\
Conversely, let us suppose that $ x_0\not\in \mathcal{F}$, then \\$x_0\not\in \left\{\prod\limits_{j=0}^{i-1}f_{A_j}^{-1}\left(-\frac{d_i}{c_i}\right), \  0\leq i\leq k-2\right\}$ which implies that $x_1,\ x_2,\cdots,x_{k-1}$ exist. Now let $n\geq k-1$ and let us suppose that $x_1, \ x_2,\cdots,x_n$ exist.
Let us prove that $x_{n+1}$ exists. Let $n=pk+r$, where $p$ and $r$ are integers such that$0\leq r\leq k-1$, then we distinguish two cases:
\medskip
\begin{itemize}[leftmargin=*]
\item Case 1: $r+1\leq k-1$. As $x_0\not\in \prod\limits_{j=0}^{r}f_{A_j}^{-1}(\mathcal{F}_{B_{r+1}})$, then $x_{r+1}=\prod\limits_{j=0}^{r}f_{A_{r-j}}(x_0)=
    f_{\prod\limits_{j=0}^{r}A_{r-j}}(x_0)\not\in\mathcal{F}_{B_{r+1}}$ so $f_{B_{r+1}}^p(x_{r+1})=f_{B_{r+1}^p A_{r}A_{r-1}\cdots A_0}(x_0)$ exist. Since
\medskip
     $f_{B_{r+1}^p A_{r}A_{r-1}\cdots A_0}(x_0)=f_{A_{r}A_{r-1}\cdots A_0B_0^p}(x_0)$, and $f_{A_{r-1}\cdots A_0B_0^p}(x_0)=x_n$ exist,
\medskip
      so $f_{A_r}(f_{A_{r-1}\cdots A_0B_0^p}(x_0))=f_{A_r}(x_n)=f_{A_n}(x_n)=x_{n+1}$ exist.
\vspace{0.2cm}
\item Case 2: $r+1=k$. As $x_0\not\in\mathcal{F}_{B_0}$ then
$f_{B_0}^{p+1}(x_0)=f_{A_{k-1}(A_{k-2}\cdots A_0B_0^p)}(x_0)$
\vspace{2mm}

\noindent
exist. Since $f_{A_{k-2}\cdots A_0B_0^p}(x_0)=x_n$ exists, so\\
\medskip
$f_{A_{k-1}(A_{k-2}\cdots A_0B_0^p)}(x_0)=f_{A_{k-1}}(f_{A_{k-2}\cdots A_0B_0^p}(x_0))=f_{k-1}(x_n)=f_n(x_n)=$
\vspace{2mm}
\noindent
$x_{n+1}$ exists.
\end{itemize}
Then the solution $(x_n)_{n\geq0}$ exist.
\end{proof}
\medskip
\begin{cor}\rm{(\cite{1})
If $k=2$, then the forbidden set of Eq.(\ref{pre}) is given by:
$$\mathcal{F}=\left\{-\frac{d_0}{c_0}\right\}\cup\mathcal{F}_{B_0}\cup
f_{A_0}^{-1}(\mathcal{F}_{B_1})$$ }
\end{cor}
\medskip
\begin{exam}\rm{
Let $k=2$ and $A_0=\left(
                 \begin{array}{cc}
                   1 & 0 \\
                  1 & 1 \\
                 \end{array}
               \right)$, \ $A_1=\left(
                 \begin{array}{cc}
                   0 & 1 \\
                  1 & 1 \\
                 \end{array}
 \right)$.\\
 We have $B_0=A_1A_0=\left(
                   \begin{array}{cc}
                     1 & 1\\
                     2 & 1 \\
                   \end{array}
                 \right)$, \ $B_1=A_0A_1=\left(
                                       \begin{array}{cc}
                                         0 & 1 \\
                                         1 & 2 \\
                                       \end{array}
                                     \right)$.\\
\noindent
 So by Corollary A.3, we have:
 $$\mathcal{F}_{B_0}=\left\{-\frac{1}{2}-\frac{(1+\sqrt{2})^{n-1}-(1-\sqrt{2})^{n-1}}{2((1+\sqrt{2})^n-(1-\sqrt{2})^n)}, \ n\geq1\right\}$$
 and
 $$\mathcal{F}_{B_1}=\left\{-2-\frac{(1+\sqrt{2})^{n-1}-(1-\sqrt{2})^{n-1}}{(1+\sqrt{2})^n-(1-\sqrt{2})^n}, \ n\geq1\right\}$$
 Therefore by Theorem 3.1, the forbidden set of Eq.(\ref{pre}) is given by:
 $$\begin{array}{ccc}
     \mathcal{F}\hspace{-0.4cm} & = \hspace{-0.7cm}& \displaystyle{ \left\{-1\right\}}\cup\left\{-\displaystyle{\frac{1}{2}}-\displaystyle{\frac{(1+\sqrt{2})^{n-1}-(1-\sqrt{2})^{n-1}}{2((1+\sqrt{2})^n-(1-\sqrt{2})^n)}}, \ n\geq1\right\} \hspace{2cm}\\
     \\
      &  & \cup\left\{\displaystyle{\frac{-(3+2\sqrt{2})(1+\sqrt{2})^{n-1}+(3-2\sqrt{2})(1-\sqrt{2})^{n-1}}{(4+3\sqrt{2})(1+\sqrt{2})^{n-1}-(4-3\sqrt{2})(1-\sqrt{2})^{n-1}}}, \ n\geq1\right\}\hspace{1cm}
   \end{array}$$}

 \end{exam}
\subsection{\bf Asymptotic behavior and stability properties of the solutions} \ \\

In this paragraph we shall decompose the Eq.(\ref{pre}) into $k$ autonomous Riccati difference equations of
type Eq.(A.1).\\
\noindent
Indeed, let $$y_p^i=x_{pk+i}, \ 0\leq i\leq k-1$$
\vspace{-3mm}
\\
Then the sequence $(x_n)_{n\geq0}$ is a solution of Eq.(\ref{pre})
with initial value $u$ if and only if for all $0\leq i\leq k-1$, $(y_n^i)_{n\geq0}$ is a solution of the equation $y_{n+1}=f_{B_i}(y_n)$ with initial value $y_0^i=x_i=\prod\limits_{j=0}^{i-1}f_{A_{i-1-j}}(u)$, where $B_i=A_{i-1}\cdots A_0A_{k-1}\cdots A_{i}$.\\

Write
$$B_i=\left(\begin{array}{cc}
\tilde{a}_i & \tilde{b}_i \\
\tilde{c}_i & \tilde{d}_i \\
\end{array}\right), \ 0\leq i\leq k-1$$
We have \\
\vspace{-3mm}
\\
$\textrm{tr}(B_0)=\textrm{tr}(B_1)=\cdots= \textrm{tr}(B_{k-1}) \ \text{and} \ \textrm{det}(B_0)=\textrm{det}(B_1)=\cdots=\textrm{det}(B_{k-1})$\\
\vspace{-2mm}
\\
We let $T=\textrm{tr}(B_0)$, \ $D=\textrm{det}(B_0)$ and $\triangle: =T^2-4D$. Then $\triangle$ is the discriminant of the equation: $X^2-TX+D=0$ called the characteristic equation of Eq.(\ref{pre}).\vspace{-2mm}\\

\vspace{1mm}
Following the notation adopted in Appendix, we denote by:
\begin{itemize}[leftmargin=*]
  \item $\lambda$, \ $\mu$ the roots of a characteristic equation if $\triangle >0$ such that $|\lambda|>|\mu|$, and $\displaystyle{\rho_i=\frac{\lambda-\tilde{d}_i}{\tilde{c}_i}}$, for all $0\leq i\leq k-1$, such that $\tilde{c}_i\neq0$.
  \item $\displaystyle{re^{\pm i\theta}}$ the complex roots of a characteristic equation if $\triangle<0.$
\end{itemize}
\vspace{3mm}
 \hspace{3mm}The following Theorem describes the asymptotic behavior of the solutions of Eq.(\ref{pre})
\medskip
\begin{thm}\rm{
\noindent
\begin{enumerate}[leftmargin=*]
\item If $T=0$ then every solution of Eq.(\ref{pre}) is 2k-periodic.
\vspace{2mm}
\item If $T\neq0$ and $\tilde{c}_i\neq0$, for all $0\leq i\leq k-1$, then,
\vspace{2mm}
\begin{description}[leftmargin=*]
\item[(2-1)] If $\triangle>0$, we have:
\begin{description}
\vspace{2mm}
\item[(i)] If $\rho_0=\rho_1=\cdots=\rho_{k-1}=\rho$, then $\rho$ is a global attractor locally asymptotically stable for
Eq.(\ref{pre}) relative to the set\\
\vspace{1mm}
 $\mathbb{R}\setminus\left(\mathcal{F}\cup\left\{\prod\limits_{j=0}^{i-1}f_{A_j}^{-1}(
\frac{\mu-\tilde{d}_{i}}{\tilde{c}_{i}}); \ 0\leq i\leq k-1\right\}\right)$.
\item[(ii)] If there exist $i\neq j$ such that $\rho_i\neq\rho_j$, then the k-cycle solution
$\{\rho_0,\rho_1,\cdots,\rho_{k-1}\}$ attracts all solutions of Eq.(\ref{pre}) with initial value outside the set
$\mathcal{F}\cup\left\{\prod\limits_{j=0}^{i-1}f_{A_j}^{-1}(
\frac{\mu-\tilde{d}_{i}}{\tilde{c}_{i}}); \ 0\leq i\leq k-1\right\}$. In addition it is stable.
\end{description}
\vspace{1mm}
\item[(2-2)]If $\triangle=0$, we have:
\begin{description}
\vspace{2mm}
\item[(i)] If $\frac{\tilde{a}_0-\tilde{d}_0}{2\tilde{c}_0}=
\frac{\tilde{a}_1-\tilde{d}_1}{2\tilde{c}_1}=
\cdots=\frac{\tilde{a}_{k-1}-\tilde{d}_{k-1}}{2\tilde{c}_{k-1}}=\rho$, then $\rho$ is a global attractor for Eq.(\ref{pre}) relative to the set $\mathbb{R}\setminus\mathcal{F}$ but unstable.
\item[(ii)] If there exist $i\neq j$ such that $\frac{\tilde{a}_i-\tilde{d}_i}{2\tilde{c}_i}
\neq\frac{\tilde{a}_j-\tilde{d}_j}{2\tilde{c}_j}$, then the k-cycle solution
$\left\{\frac{\tilde{a}_0-\tilde{d}_0}{2\tilde{c}_0},\frac{\tilde{a}_1-\tilde{d}_1}{2\tilde{c}_1},
\cdots,\frac{\tilde{a}_{k-1}-\tilde{d}_{k-1}}{2\tilde{c}_{k-1}}\right\}$
attract all solutions of Eq.(\ref{pre}) with initial value outside the set $\mathcal{F}$, but unstable.
\end{description}
\vspace{2mm}
\item[(2-3)] If $\triangle<0$, we have:
\begin{description}
\vspace{2mm}
\item[(i)]  If $\theta=\frac{p}{q}\pi, \ p \in \mathbb{Z}, \ q \in \mathbb{N}\setminus\{0,1\}, \ \text{with}
\ gcd(p,q)=1$, then all solutions of Eq.(\ref{pre}) are $kq$-periodic.
\vspace{2mm}
\item[(ii)]If $\theta\not\in\pi\mathbb{Q}$, then every solution of Eq.(\ref{pre}) is dense in $\mathbb{R}$.
\end{description}
\end{description}
\end{enumerate}}
\end{thm}
\medskip

\begin{proof} \begin{description}[leftmargin=*]
\item[(1)] Suppose that $T=0$, then by Proposition A.4, we have: for all $0\leq i\leq k-1$,
the solution $(y_n^i)_{n\geq0}$ of equation $y_{n+1}=f_{B_i}(y_n)$ is 2-periodic.
Now let $n \ \in \mathbb{N}$, and $(x_n)_{n\geq0}$ a solution of Eq.(\ref{pre}). As $n=pk+r$ where $p, \ r$ are integers such that $0\leq r\leq k-1$, then
$x_{n+2k}=x_{pk+r+2k}=x_{(p+2)k+r}=y_{p+2}^r=y_p^r=x_n$, therefore the solution $(x_n)_{n\geq0}$ is 2k-periodic.
\medskip

\item[(2-1)-(i)] Suppose that $T\neq0, \ \tilde{c}_i\neq0, \ \text{for all} \ 0\leq i\leq k-1, \ \triangle>0, \ \text{and}$\\ $\rho_0=\rho_1=\cdots=\rho_{k-1}=\rho$, then by Proposition A.4 we have: for all $0\leq i\leq k-1$, and for all solution $(y_n^i)_{n\geq0}$ of equation $y_{n+1}=f_{B_i}(y_n)$, \
with initial value in $\mathbb{R}\setminus(\mathcal{F}\cup\{\frac{\mu-d_i}{c_i}\})$ we have:
\ $\displaystyle{\lim_{n\rightarrow+\infty}y_n^i=\rho}$, so for all $u\ \in\ \mathbb{R}\setminus\mathcal{F}$ such that
$u\neq\frac{\mu-d_0}{c_0}, \ x_1(u)\neq \frac{\mu-d_1}{c_1},\cdots, \ x_{k-1}(u)\neq \frac{\mu-d_{k-1}}{c_{k-1}}$,
the solution $(x_n(u))_{n\geq0}$ of Eq.(\ref{pre}) with initial value $u$ converges to $\rho$. Therefore the point
$\rho$ is a global attractor for Eq.(\ref{pre}) relative to the set\\
 $\mathcal{G}=:\mathbb{R}\setminus\left(\mathcal{F}\cup\left\{\prod\limits_{j=0}^{i-1}f_{A_j}^{-1}(
\frac{\mu-\tilde{d}_{i}}{\tilde{c}_{i}}); \ 0\leq i\leq k-1\right\}\right)$.\\
\noindent
Now we show that $\rho$ is stable. Let $\epsilon>0$. Since for all $0\leq i\leq k-1$, the point $\rho$ is stable
for equation $y_{n+1}=f_{B_i}(y_n)$ \ (Proposition A.4), there exist $\alpha_i>0$ such that:
$\forall\ y_0^i \in\ \mathbb{R}\setminus(\mathcal{F}\cup\{\frac{\mu-d_i}{c_i}\})$, such that: $|y_0^i-\rho|<\alpha_i$,
and for all solution $(y_n^i)_{n\geq0}$ of equation $y_{n+1}=f_{B_i}(y_n)$, we have $|y_n^i-\rho|<\epsilon$,
for all
$n\geq0$. In addition, as $x_1, \ x_2,\cdots, \ x_{k-1}$ are continuous at $\rho$,
and $ x_1(\rho)= x_2(\rho)=\cdots=x_{k-1}(\rho)=\rho$, then for all $1\leq i\leq k-1$, there exists
$\beta_i>0$ such that: for all $u$ with $|u-\rho|<\beta_i$, we have $|x_i(u)-\rho|<\alpha_i$. Finally let us take
$\eta=\min\{\alpha_0, \ \beta_1, \ \beta_2,\cdots, \ \beta_{k-1}\}$ and $u \in \mathcal{G}$ such that $|u-\rho|<\eta$, \ $(x_n(u))_{n\geq0}$ a solution of Eq.(\ref{pre}) with initial value $u$, and $n$ is a integer, since $n=pk+i$ where $p, \ i$ are integers such that $0\leq i\leq k-1$, then:
\begin{itemize}[leftmargin=*]
\item If $i=0$, as $|u-\rho|<\eta\leq\alpha_0$, then $|x_n-\rho|=|y_p-\rho|<\epsilon$.
\item If $1\leq i\leq k-1$, as $|u-\rho|<\eta\leq\beta_i$, then $|x_i(u)-\rho|=|y_0^i(u)-\rho|<\alpha_i$, which implies
that $|x_n-\rho|=|y_p^i-\rho|<\epsilon$.
\end{itemize}
Therefore the point $\rho$ is stable.
\medskip

\item[(2-1)-(ii)] We have $\rho_{i+1}=f_{A_i}(\rho_i), \ \text{for all} \ 0\leq i\leq k-2, \ \text{and} \ f_{B_i}(\rho_i)=\rho_i$,\\ $\text{for all} \ 0\leq i\leq k-1$. Then $\mathcal{C}:=\{\rho_0, \ \rho_1,\ldots,
\ \rho_{k-1}\}$ is a $k$-cycle solution of Eq.(\ref{pre}). Now let $u \in \mathcal{G}$, and $(x_n(u))_{n\geq0}$ a
solution of Eq.(\ref{pre}) with initial value $u$. Then by Proposition A.4,
we have: $\displaystyle{\lim_{n\rightarrow+\infty}x_{nk+i}=\rho_i, \ \text{for all}}$\\ $0\leq i \leq k-1$, so the cycle  $\mathcal{C}$ attract all solutions of Eq.(\ref{pre}) with initial value in $\mathcal{G}$. Similarly to
the Case 2-1-(i), we easily prove that $\mathcal{C}$ is stable (in
the sense cited in \cite{2} (definition 7)). Indeed let $\epsilon>0$,
and set \\$\eta=\min\{\alpha_0, \ \beta_1, \ \beta_2,\cdots, \ \beta_{k-1}\}$ as in 2-1-(i). Then for all
$u \in \mathcal{G}$ such that $|u-\rho|<\eta$, for all $n \in \mathbb{N}$, and for all $0\leq i\leq k-1$, we have: $|x_{nk+i}(u)-\rho_i|<\epsilon$. Therefore $\mathcal{C}$ is stable.
\medskip

\item[(2-2)-(i)] Assume that $T\neq0, \ \tilde{c}_i\neq0, \ \text{for all} \ 0\leq i\leq k-1, \ \triangle>0, \ \text{and}$\\
$\frac{\tilde{a}_0-\tilde{d}_0}{2\tilde{c}_0}=
\frac{\tilde{a}_1-\tilde{d}_1}{2\tilde{c}_1}=
\cdots=\frac{\tilde{a}_{k-1}-\tilde{d}_{k-1}}{2\tilde{c}_{k-1}}=\rho$, then similarly to the case 2-1-(i),
we can show that all solutions of Eq.(\ref{pre}) converge to $\rho$.
The instability of $\rho$ follows from the fact that $\rho$ is unstable as equilibrium point of equations $y_{n+1}=f_{B_i}(y_n)$ \ (Proposition A.4).
\medskip

\item[(2-2)-(ii)]The study of convergence of solutions is similar to the Case 2-1-(ii), and the instability character is similar to the Case 2-2-(i).
\medskip

\item[(3-1)] Assume that $\triangle<0, \ \text{and} \ \theta=\frac{p}{q}\pi, \ p \in \mathbb{Z},
\ q \in \mathbb{N}\setminus\{0,1\}$\ such that \ $gcd(p,q)=1$.
Then by Proposition A.4, we have: for all $0\leq i\leq k-1$, the solution $(y_n^i)_{n\geq0}$ of
equation $y_{n+1}=f_{B_i}(y_n)$ is q-periodic. Now let $n \ \in \mathbb{N}$, and $(x_n)_{n\geq0}$ a solution of
Eq.(\ref{pre}), as $n=pk+r$ where $p, \ r$ are integers such that $0\leq r\leq k-1$. Then
$x_{n+2qk}=x_{pk+r+2qk}=x_{(p+2q)k+r}=y_{p+2q}^r=y_p^r=x_n$, therefore the solution $(x_n)_{n\geq0}$ is $qk$-periodic.
\medskip

\item[(3-2)] This follows from the fact that $(x_{nk+i})_{n\geq0}$ is dense in $\mathbb{R}$ for all\\ $0\leq i\leq k-1$
(see Proposition A.4).
\end{description}
\end{proof}

\begin{exam}\rm{
\noindent
\begin{enumerate}[leftmargin=*]
\item Let $k=4$ and we let $A_0=\left(
             \begin{array}{cc}
               1 & 0 \\
               1 & -1 \\
             \end{array}\right), \
A_1=\left(\begin{array}{cc}
        1 & 0 \\
        -1 & 1 \\
      \end{array}\right)$,\\
$A_2=\left(\begin{array}{cc}
        1 & 1 \\
       -1 & 1 \\
      \end{array}\right) \ \text{and} \
A_3=\left(\begin{array}{cc}
        1 & -1 \\
        1 & 1 \\
      \end{array}
    \right)$.\\
 We have $B_0=A_3A_2A_1A_0=\left(\begin{array}{cc}
                                      2 & 0 \\
                                      0& -2 \\
                                    \end{array}\right)$. Since $T=0$,
so by Theorem 3.4,(1), every solution of Eq.(\ref{pre}) is 8-periodic.
\vspace{2mm}
\item Let $k=2$ and we let $A_0=\left(
             \begin{array}{cc}
               1 & 1 \\
               1 & 0 \\
             \end{array}
           \right) \ \text{and} \
A_1=\left(
      \begin{array}{cc}
        1 & -1 \\
        -1 & 2 \\
      \end{array}
    \right)$. We have $B_0=B_1=\left(
                                    \begin{array}{cc}
                                      0 & 1 \\
                                      1& -1 \\
                                    \end{array}
                                  \right)$, \ $T=-1, \ \triangle=5>0, \ \lambda=-\frac{1+\sqrt{5}}{2}, \\ \mu=\frac{-1+\sqrt{5}}{2}, \ \rho_0=\rho_1=\frac{1-\sqrt{5}}{2}$.
Then by Theorem 3.4,(2-1),(i), the point $\frac{1-\sqrt{5}}{2}$ is a global attractor locally asymptotically stable for Eq.(\ref{pre}) relative to set
 $\mathbb{R}\setminus\left(\mathcal{F}\cup\left\{\frac{1+\sqrt{5}}{2}\right\}\right)$.
\vspace{2mm}
\item Let $k=2$ and we let $A_0=\left(
             \begin{array}{cc}
               1 & 0 \\
               1 & 1 \\
             \end{array}
           \right) \ \text{and} \
A_1=\left(
      \begin{array}{cc}
        0 & 1 \\
        1& 1 \\
      \end{array}
    \right)$. We have $B_0=\left(
                                    \begin{array}{cc}
                                     1 & 1 \\
                                      2& 1 \\
                                    \end{array}
                                  \right)$,
$B_1=\left(
                                    \begin{array}{cc}
                                     0 & 1 \\
                                      1& 2 \\
                                    \end{array}
                                  \right)$,
\ $T=2, \ \triangle=8>0, \ \lambda=1+\sqrt{2}, \\ \mu=1+\sqrt{2}, \ \rho_0=\frac{\sqrt{2}}{2}, \ \rho_1=\sqrt{2}-1$. Then by Theorem 3.4,(2-1),(ii), the 2-cycle $\left\{\frac{\sqrt{2}}{2}, \ \sqrt{2}-1\right\}$, is a global attractor locally asymptotically stable for Eq.(\ref{pre}) relative to set
 $\mathbb{R}\setminus\left(\mathcal{F}\cup\left\{-\frac{\sqrt{2}}{2}\right\}\right)$.
\end{enumerate}
}\end{exam}
\section{\bf The special case: $b_n=0$ }
In this section we suppose that $b_n=0$, for all integers $n$. By dividing by $d_n$, the Eq.(\ref{pre}) is reduced to the equation:
\begin{equation}
x_{n+1}=\frac{a_nx_n}{c_nx_n+1}, \ a_nc_n\neq0, \ n=0,1,\cdots \label{spre}
\end{equation}
This equation was partially investigated by Clarck, M.E. and Gross, L.J. (\cite{8}) when the parameters $a_n, \ c_n$ and initial
value are positive, and by S. Stevi\'c (\cite{7}) in view to describing the forbidden set when the parameters for $a_n, \ c_n$ and initial value are arbitrary reals.\\
Based on the theoretical support that we developed in Section 3, we will give in this section a complete study of
Eq.(\ref{spre}), which extends and improves the results of M.E. Clarck, and L.J. Gross,
Moreover we precise the forbidden set given by S. Stevi\'c.\\

By induction and a basic matrix calculation we can show the following:
\medskip
\begin{lem}\rm{ Let $ A_{n}
=\left(\begin{array}{cc}
a_n & 0 \\
c_n & 1 \\
\end{array}\right), \ n\geq0$. Then
\noindent
\begin{enumerate}[leftmargin=*]
\item For all $0\leq i\leq k-1$, we have:
$$B_i=\left(\begin{array}{cc}
\tilde{a}& \ \ 0 \\
\\
\tilde{c}_i & \ \ 1 \\
\end{array}\right)$$
 where
 $$\tilde{a}=\prod\limits_{j=0}^{k-1} a_j \ \ \text{and} \ \ \tilde{c}_i=\prod\limits_{j=i}^{k-1} a_j\sum\limits_{l=0}^{i-1}c_l\prod_{r=0}^{l-1} a_r+\sum\limits_{s=i}^{k-1}c_s\prod\limits_{t=i}^{s-1} a_t$$
\item $$ B_i^n=\left(\begin{array}{cc}
           \tilde{a}^n & 0 \\
\\
            \displaystyle{\frac{\tilde{a}^n-1}{\tilde{a}-1}\tilde{c_i}} & 1 \\
 \end{array} \right), \ \text{for all} \ 0\leq i\leq k-1, \ \text{and all} \ n\geq0$$
\vspace{1mm}
with the convention $\displaystyle{\frac{\tilde{a}^n-1}{\tilde{a}-1}}=n, \ \text{if} \ \tilde{a}=1$.
\vspace{1mm}
\item The forbidden set $\mathcal{F}_{B_i}$ is given by:
$$\left\{\begin{array}{ccc}
      \left\{-\left(\displaystyle{\frac{\tilde{a}^n-1}{\tilde{a}-1}}\tilde{c_i}\right)^{-1}, \ n\geq1\right\}&, \ \text{if}&\tilde{c}_i\neq0 \\
      \\
      \emptyset &, \ \text{if}&\tilde{c}_i=0
    \end{array}\right.$$
\end{enumerate}
}\end{lem}
\medskip
\noindent
Therefore, we have the following:
\medskip
\begin{prop}\rm{
The forbidden of Eq.(\ref{spre}) is given by:\\
\vspace{1mm}

\noindent
$
\begin{array}{c}\mathcal{F}=\left\{-\left(\sum\limits_{j=0}^{i-1}c_j
\prod\limits_{r=0}^{j-1} a_r+c_i\prod\limits_{j=0}^{i-1} a_j\right)^{-1}; \ 0\leq i\leq k-2: \ \sum\limits_{j=0}^{i-1}c_j
\prod\limits_{r=0}^{j-1} a_r+c_i\prod\limits_{j=0}^{i-1} a_j\neq0\right\} \hspace{1.5cm}\\
\\
\bigcup\bigcup\limits_{i \in J_k}\left\{-\left(\sum\limits_{j=0}^{i-1}c_j\prod\limits_{r=0}^{j-1} a_r+\frac{\tilde{a}^n-1}{\tilde{a}-1}\tilde{c_i}\prod\limits_{j=0}^{i-1} a_j\right)^{-1}; \ n\geq1: \ \sum\limits_{j=0}^{i-1}c_j\prod\limits_{r=0}^{j-1} a_r+\frac{\tilde{a}^n-1}{\tilde{a}-1}\tilde{c_i}\prod\limits_{j=0}^{i-1} a_j\neq0\right\}\end{array}$\\
\vspace{5mm}

where $J_k=\{i, \ 0\leq i\leq k-1: \ \tilde{c}_i\neq0\}$.
}\end{prop}
\medskip
\begin{proof}\rm{
By Proposition 2.4, we have:\\
\vspace{2mm}
\noindent
$\prod\limits_{j=0}^{i-1}f_{A_{i-1-j}}=f_{\prod\limits_{j=0}^{i-1}A_{i-1-j}}=f_{\bar{A}_i}$
and $(f_{\bar{A}_i})^{-1}(X)=f_{\bar{A}_i^{-1}}(X)$, for all $X\subset\mathbb{R}$
\vspace{2mm}
\noindent
and for all $0\leq i\leq k-1$.\\
\vspace{2mm}
\noindent
Since $\prod\limits_{j=0}^{i-1}f_{A_j}^{-1}(X)=
\left(\prod\limits_{j=0}^{i-1}f_{A_{i-1-j}}\right)^{-1}(X), \ \text{for all} \ X\subset\mathbb{R}$, then
\vspace{-1mm}
$$\prod\limits_{j=0}^{i-1}f_{A_j}^{-1}(X)=f_{\bar{A}_i^{-1}}(X), \ \text{for all} \ X\subset\mathbb{R} \ \text{and for all} \ 0\leq i\leq k-1$$
\vspace{1mm}
\noindent
Now, by Proposition 2.6 the matrices $\bar{A}_i$ are given by:
$$\bar{A}_n=\left(\begin{array}{cc}
\prod\limits_{i=0}^{n-1} a_i & 0 \\
                         \\
\sum\limits_{i=0}^{n-1}c_i\prod\limits_{j=0}^{i-1}a_j & 1\\
\end{array}\right)$$
so
$$(\bar{A}_n)^{-1}=\left(\prod\limits_{i=0}^{n-1} a_i\right)^{-1}\left(\begin{array}{cc}
1 & 0 \\
                         \\
-\sum\limits_{i=0}^{n-1}c_i\prod\limits_{j=0}^{i-1}a_j & \prod\limits_{i=0}^{n-1} a_i\\
\end{array}\right)$$
Finally, by Theorem 3.1,  we deduce Proposition 4.2.
}\end{proof}
\medskip

Now using Theorem 3.4, we shall give a complete description of the asymptotic behavior of solutions of Eq.(\ref{spre}). The following result extends and improves the result of Clarck, M.E. and Gross, L.J. (\cite{8}), (see also \cite{5}, Appendix A, Theorem A.5).
\begin{prop} \rm{
\noindent
\begin{enumerate}[leftmargin=*]
\item If $\tilde{a}=-1$, then every solution of Eq.(\ref{spre}) is 2k-periodic.
\vspace{3mm}
\item If $\tilde{a}\neq-1$, we have:
\vspace{3mm}
\begin{description}
\item[(i)] $\mid\tilde{a}\mid<1$, then 0 is a global attractor locally asymptotically stable for Eq.(\ref{spre}) relative to set\\
 $\mathbb{R}\setminus(\mathcal{F}\cup\{\prod_{j=0}^{i-1}f_{A_i}^{-1}
 (\frac{\tilde{a}-1}{\tilde{c}_i}); \ 0\leq i\leq k-1: \ \tilde{c}_i\neq0\})$.
\vspace{3mm}
\item[(ii)] If $\tilde{a}=1$, and $\tilde{c}_i\neq0$ for all $0\leq i\leq k-1$, then 0 is a global attractor for Eq.(\ref{spre}) relative to the set $\mathbb{R}\setminus\mathcal{F}$ but is unstable.
\vspace{3mm}
\item[(iii)] If $\mid\tilde{a}\mid>1$, and $\tilde{c}_i\neq0$ for all $0\leq i\leq k-1$, we have:
\vspace{3mm}
\begin{itemize}
\item If $\tilde{c}_0=\tilde{c}_1=\cdots=\tilde{c}_{k-1}$, then $\tilde{a}$ is a global attractor locally asymptotically stable for Eq.(\ref{spre}) relative to set
 $\mathbb{R}\setminus\mathcal{F}$.
\vspace{3mm}
\item If there exist $i\neq j$ such that $\tilde{c}_i\neq\tilde{c}_j$, then the k-cycle solution $\{\frac{\tilde{a}-1}{\tilde{c}_0},\frac{\tilde{a}-1}{\tilde{c}_1},
    \cdots,\frac{\tilde{a}-1}{\tilde{c}_{k-1}}\}$ attract
 all solutions of Eq.(\ref{spre}) with initial value outside the set
 $\mathcal{F}\cup\{0\}$, in addition it is stable.
\end{itemize}
\item[(iv)] If $\mid\tilde{a}\mid\geq1$, and there exist $i$ such that $\tilde{c}_i=0$, then any nonzero solution of Eq.(\ref{spre}) is oscillating.
\end{description}
\end{enumerate}
}\end{prop}
\begin{proof}\rm{
The parameters describing Eq.(\ref{spre}) are: $$D=\tilde{a}, \ \ T=\tilde{a}+1, \ \  \triangle=T^2-4D=(\tilde{a}-1)^2$$
So, by using the same notation adopted in section 3 we have:
$$\lambda=\left\{\begin{array}{ccc}
   1, &\text{if} & \mid\tilde{a}\mid<1 \\
    \tilde{a}, & \text{if} & \mid\tilde{a}\mid>1
  \end{array}\right., \ \
 \mu=\left\{\begin{array}{ccc}
   \tilde{a}, & \text{if} & \mid\tilde{a}\mid<1 \\
   1, & \text{if} & \mid\tilde{a}\mid>1
  \end{array}\right.$$
  Now apply Theorem 3.4, we get (1), (2)-(i),(ii),(iii).\\
  Finally suppose that $\tilde{a}\neq-1 \ \mid\tilde{a}\mid\geq1$, and there exists $i$ such that $\tilde{c}_i=0$.
  Let $(x_n)_{n\geq0}$ be nonzero solution of Eq.(\ref{spre}), since $\tilde{c}_i=0$, then $$x_{nk+i}=\tilde{a}^nx_i\ _ { \overrightarrow{n\rightarrow+\infty}}
  \left\{\begin{array}{ccc}
     \infty, & \text{if} & \mid\tilde{a}\mid>1 \\
      x_i, & \text{if} & \tilde{a}=1
    \end{array}\right.
  $$
 Now we have \\
 \\
  $\left(\begin{array}{c}
            \tilde{c}_0 \\
            \tilde{c}_1 \\
            \vdots \\
          \tilde{c}_{k-1} \\
\end{array}\right)
=M\left(\begin{array}{c}
            c_0 \\
            c_1 \\
            \vdots \\
          c_{k-1} \\
\end{array}\right)$\\
\\
where $M$ is the matrix
$$M=\left(
      \begin{array}{cccccc}
        1 & a_0 & a_0a_1 &\cdots & a_0\cdots a_{k-3} & a_0\cdots a_{k-2} \\
        a_1\cdots a_{k-1} & 1 & a_1 & \cdots & a_1\cdots a_{k-3}& a_1\cdots a_{k-2}\\
        a_2\cdots a_{k-1} & a_2\cdots a_{k-1}a_0 & 1 &\cdots &a_2\cdots a_{k-3} &  a_2\cdots a_{k-2} \\
        \vdots & \vdots & \vdots & \vdots &\vdots & \vdots \\
        a_{k-2}a_{k-1} & a_{k-2}a_{k-1}a_0 & a_{k-2}a_{k-1}a_0a_1 & \cdots & 1 & a_{k-2} \\
        a_{k-1} &  a_{k-1}a_0 &  a_{k-1}a_0a_1 & \cdots & a_{k-1}a_0\cdots a_{k-3} & 1 \\
      \end{array}
    \right)$$
\\
and $\det(M)=(\tilde{a}-1)^{k-1}\neq0$. Then there exist $j\neq i$ such that $\tilde{c}_i\neq0$. This implies by Lemma 4.1, that
$$x_{nk+j}=\frac{\tilde{a}^nx_j}{\frac{\tilde{a}^n-1}{\tilde{a}-1}x_j+1}\ _ { \overrightarrow{n\rightarrow+\infty}}\left\{\begin{array}{ccc}
     \tilde{a}, & \text{if} & \mid\tilde{a}\mid>1 \\
      0, & \text{if} & \tilde{a}=1
    \end{array}\right.$$
Therefore the solution $(x_n)_{n\geq0}$ is oscillating.
}\end{proof}
\medskip
\begin{cor}\rm{(\cite{5}, Appendix A, Theorem A.5)
Assume that $a_n, \ c_n$ and the initial value are positive. Then:
\begin{enumerate}[leftmargin=*]
\item If $\tilde{a}\leq1$, then the zero solution of Eq.(\ref{spre}) is a global attractor of all positive solutions of this equation.
\item If $\tilde{a}>1$, then the solution $(\bar{x}_n)_{n\geq0}$ of Eq.(\ref{spre}) with initial value $\bar{x}_0=\frac{\tilde{a}-1}{\tilde{c_0}}$ is globally asymptotically stable k-periodic solution of this equation.
\end{enumerate}
}\end{cor}
\begin{proof}\rm{
Under the hypothesis of Corollary 4.4, we have, $\mathcal{F}\subset]-\infty,0[$, $\tilde{a}>0$ and $\tilde{c}_i>0$, for all $0\leq i\leq k-1$, therefore, the Corollary 4.4 follows immediately from Proposition 4.3,(2).
}\end{proof}
\medskip
\section{\bf A class of non-autonomous Riccati equations with coefficients not necessarily periodic }
Let us assume that $a_n=1$. Then the Eq.(\ref{spre}) is reduced to the equation:
\begin{equation}
x_{n+1}=\frac{x_n}{c_nx_n+1}, \ c_n\neq0, \ n=0,1,\cdots \label{snore}
\end{equation}
In this section we shall study Eq.(\ref{snore}) when $(c_n)_{n\geq0}$ is not necessarily periodic in view to see the difference between the periodic and non-periodic cases. The following proposition describes the behavior of solutions of Eq.(\ref{snore}) and determine the forbidden set.
In fact, we have the following:
\begin{prop}\rm{
\noindent
\begin{enumerate}[leftmargin=*]
\item The forbidden set of Eq.(\ref{snore}) is given by:
$$\mathcal{F}=\left\{\left(\sum\limits_{i=0}^{n-1}c_i\right)^{-1}; \ n\geq1: \
\ \sum\limits_{i=0}^{n-1}c_i\neq0\right\}$$
\item If $\displaystyle{\lim_{n\rightarrow+\infty}c_n=c\neq0}$, then 0 attracts all solutions of Eq.(\ref{snore}) but is not stable.
\vspace{3mm}
\item If $\sum\limits_{n\geq0}c_n$ converges, then for any solution $(x_n)_{n\geq0}$ of Eq.(\ref{snore}), we have:
    $$\lim_{n\rightarrow+\infty}x_n=\frac{x_0}{x_0\sum\limits_{n\geq0}c_n+1}$$
    Particularly, if $\sum\limits_{n\geq0}c_n=0$ then $\displaystyle{\lim_{n\rightarrow+\infty}x_n=x_0}$.
\end{enumerate}
}\end{prop}
\medskip
\begin{proof}\rm{
 Let $(x_n)_{n\geq0}$ be a solution of Eq.(\ref{snore}), then by Corollary 2.5 and Proposition 2.6 we have:
$$x_n=\frac{x_0}{x_0\sum\limits_{i=0}^{n-1}c_i+1}, \ n=0,1,\cdots \ (*)$$
It follows that
\begin{enumerate}[leftmargin=*]
\item The forbidden set of Eq.(\ref{snore}) is given by:
$$\mathcal{F}=\left\{\left(\sum\limits_{i=0}^{n-1}c_i\right)^{-1}; \ n\geq1: \
\ \sum\limits_{i=0}^{n-1}c_i\neq0\right\}$$
\item Since $\displaystyle{\lim_{n\rightarrow+\infty}c_n=c\neq0}$, then $\sum\limits_{i=0}^{n-1}c_i\ _ { \overrightarrow{n\rightarrow+\infty}}\infty$, so $\displaystyle{\lim_{n\rightarrow+\infty}x_n=0}$.\\Let us suppose that 0 is stable and let $\epsilon>0$, then there exists $\alpha>0$ such that: for any initial value $u \in \mathbb{R}\setminus\mathcal{F}$ with $\mid u\mid<\alpha$, and for any solution $(x_n)_{n\geq0}$ with initial value $u$, we have:
$$\mid x_n\mid=\left|\frac{u}{u\sum\limits_{i=0}^{n-1}c_i+1}\right|<\epsilon \  (**)$$
Let $S_n=\sum\limits_{i=0}^{n-1}c_i$ and $u_n=-\displaystyle{\frac{1}{2}(\frac{1}{S_n}+\frac{1}{S_{n+1}})}$. As $\sum\limits_{i=0}^{n-1}c_i\ _ { \overrightarrow{n\rightarrow+\infty}}\infty$, then there exists a integer $n_0$ such that for all $n\geq n_0$, $u_n$ exists and satisfies $|u_n|<\alpha$. So by $(**)$, we get
$$\left|\frac{u_n}{u_n\sum\limits_{i=0}^{n-1}c_i+1}\right|=
\left|\frac{2}{c_n}+\frac{2}{S_n}\right|<\epsilon, \ \text{for all} \ n\geq n_0$$
So $\frac{2}{|c|}=0$, which is a contradiction, we conclude that 0 is not stable.
\item If $\displaystyle{\sum_{n\geq0}c_n}$ converges we have: $\displaystyle{\lim_{n\rightarrow+\infty}x_n=\frac{x_0}{x_0\sum\limits_{n\geq0}c_n+1}}$.
\end{enumerate}

}\end{proof}
\appendix{\[ \text{\textbf{Appendix}}\]}
\medskip

\noindent
\hspace{2mm} In this section we give a new proof of some results related to the asymptotic behavior of solutions of the first-order Riccati
autonomous difference equation defined by:
\begin{equation*}%
(\text{A.1}) \hspace{1cm} x_{n+1}=\frac{ax_n+b}{cx_n+d}; \ c\neq0 \ \text{and} \ ad-bc\neq0 \label{re}
\end{equation*}%
\noindent Set $A=\left(\begin{array}{cc}
a & b\\
c & d \\
\end{array}\right)$. Using the tools of linear algebra we can explicit $A^n$ and therefore the solution of Eq.(A.1)
and the associated forbidden set as follows:\\
\noindent
\medskip

Let $P_A(x)=x^2-\textrm{tr}(A)x+\textrm{det}(A)$ the characteristic polynomial of $A$ and denote
by: $\triangle_A=tr^2(A)-4det(A)$ their discriminant, where $tr(A)= a+c$ and $det(A)= ad-bc$. Then we have:\\

\Lemma[A.2.]
\vspace{-0.5cm}\begin{enumerate}[leftmargin=*]
  \item If $\triangle_A\neq0$, we have:
  $$A^n=\frac{\lambda_1^n-\lambda_2^n}{\lambda_1-\lambda_2}A-
  \frac{\lambda_1\lambda_2(\lambda_1^{n-1}-\lambda_2^{n-1})}{\lambda_1-\lambda_2}I$$
  where $\lambda_1, \ \lambda_2$ are the complex roots of $P_A$ and $I$ is the identity matrix.
  \item If $\triangle_A=0$, we have:
  $$A^n=\lambda^{n-1}[nA+\lambda(1-n)I]$$
  where $\lambda$ is the unique root of $P_A$.
\end{enumerate}
\medskip
By Lemma A.2 and Corollary 2.5, we have:\\

\Corollary [A.3.]
\vspace{-0.5cm}\begin{enumerate}[leftmargin=*]
  \item If $\triangle_A\neq0$, then the forbidden set of Eq.(A.1) is given by:
  $$\mathcal{F}=\left\{-\frac{d}{c}+ \frac{\lambda_1\lambda_2(\lambda_1^{n-1}-\lambda_2^{n-1})}{(\lambda_1^n-\lambda_2^n)c}; \ n\geq1: \ \lambda_1^n\neq\lambda_2^n\right\}$$
\vspace{1mm}
  \item If $\triangle_A=0$, then the forbidden set of Eq.(A.1) is given by:
  $$\mathcal{F}=\left\{-\frac{d}{c}+ \frac{\lambda(n-1)}{nc}, \ n\geq1\right\}$$
\end{enumerate}

\medskip

We conclude this section by studying the asymptotic behavior of the solutions of the equation Eq.(A.1).\\
\noindent

If $\triangle_A>0$ and $tr(A)\neq0$, then the characteristic polynomial $P_A$ have two real roots, noted
$\lambda, \ \mu$ with $|\lambda|>|\mu|$. If $\triangle_A<0$, we denote by $re^{\pm i\theta}$, $\theta \in \ ]0,\pi[$, the complex roots of $P_A$.\\

\Proposition[A.4.]
\vspace{-0.5cm}\begin{enumerate}[leftmargin=*]
\item If $tr(A)=0$, then all solutions of Eq.(A.1) are $2$-periodic.
\item If $tr(A)\neq0$, we have:
\begin{description}[leftmargin=*]
\item[(i)]If $\triangle_A>0$, then the equilibrium point $\rho=\frac{\lambda-d}{c}$ of Eq.(A.1)
is a globally asymptotically stable relative to the set $ \mathbb{R}\backslash\left(\mathcal{F}\cup\left\{\frac{\mu-d}{c}\right\}\right)$.
\item[(ii)] If $\triangle_A=0$, then the point $\rho=\frac{a-d}{2c}$ attracts all solutions of Eq.(A.1) but is unstable.
\item[(iii)]If $\triangle_A<0$, $\theta=\frac{p}{q}\pi$ with $p \ \in \ \mathbb{Z}, \ q \ \in \mathbb{N}\setminus\{0,1,2\}$ and \ $gcd(p,q)=1$, then all solutions of Eq.(A.1) are $q$-periodic.
\item[(iv)] If $\triangle_A<0$ and $\theta\not\in\pi\mathbb{Q}$, then every solution of Eq.(A.1) is dense in $\mathbb{R}$.
\end{description}
\end{enumerate}
\medskip
\textit{Proof.}
\begin{description}[leftmargin=*]
\item[(1)] If $tr(A)=0$, then $A^2=\alpha I$ and for all $n\in\mathbb{N}$,
  we have: $A^{2n}=\alpha^nI$ and $A^{2n+1}=\alpha^nA$. Then the solution $(x_n)_{n\geq0}$ of Eq.(A.1) is given by: $x_{2n}=f_{A^{2n}}(x_0)=x_0$ and $x_{2n+1}=f_{A^{2n+1}}(x_0)=x_1$, therefore $(x_n)_{n\geq0}$ is 2-periodic.
\item[(2)-(i)] If $tr(A)\neq0$ and $\triangle_A>0$.
  Let $P=\left(\begin{array}{cc}
  p_{11} & p_{12}\\
  p_{21} & p_{22} \\
\end{array}\right)$ be an invertible matrix such that $A=PDP^{-1}$,
  where $D=\left(\begin{array}{cc}
  \mu & 0\\
   0 & \lambda \\
\end{array}\right)$.\\
We have $P^{-1} = \frac{1}{det(P)}\left(\begin{array}{cc}
                                                   p_{22} & -p_{12}\\
                                                  -p_{21} & p_{11} \\
                                                 \end{array}
                                               \right)$.
As $\left(\begin{array}{c}
              p_{11} \\
              p_{21} \\
\end{array}\right)$ and
$\left(\begin{array}{c}
              p_{12} \\
              p_{22} \\
\end{array}\right)$
are the eigenvectors associated to the eigenvalues $\lambda$ and $\mu$, then the coefficients $p_{ij}$ satisfy:
$$\left\{\begin{array}{c}
cp_{11}+(d-\mu)p_{21}=0 \\
 cp_{12}+(d-\lambda)p_{22}=0
\end{array}\right.$$
hence, $p_{21}\neq0$ and $p_{22}\neq0$.
Then for all $x\in\mathbb{R}\setminus(\mathcal{F}\cup\{\frac{p_{11}}{p_{21}}=\frac{\mu-d}{c}\})$, the solution $(x_n(x))_{n\geq0}$ of  Eq.(A.1) with initial value $x$ is given by:
$$x_n=f_{A^n}(x)=f_P\left(\left(\frac{\lambda}{\mu}\right)^n f_{P^{-1}}(x)\right)$$
\noindent
Since $\displaystyle{\lim_{n\rightarrow+\infty}(\frac{\mu}{\lambda})^nf_P^{-1}(x)=0}$, and $f_P$ is defined in $0$ ($p_{22}\neq0$), then
$$\lim_{n\rightarrow+\infty}x_n=f_P(0)=\frac{p_{12}}{p_{22}}=\frac{\lambda-d}{c}$$
\noindent
Therefore the point $\rho=\frac{\lambda-d}{c}$ is a global attractor of Eq.(A.1) with basin
of attraction $\mathbb{R}\setminus(\mathcal{F}\cup\{\frac{\mu-d}{c}\})$.\\
Now we will show that $\rho$ is stable.\\
\noindent
Let $\epsilon>0$, as $f_P$ is continuous at $0$, then there exists $\alpha>0$
such that $\forall \ y \ \in \ ]-\alpha, \alpha[$, we have:  $|f_P(y)-f_P(0)|=|f_P(y)-\rho|\leq\epsilon$. Since $f_P$
is continuous at $\rho$, then there exits $\eta>0$, such that $\forall \ x \ \in \ ]\rho-\eta, \rho+\eta[$, we have:
$f_{P^{-1}}(x) \ \in \ ]-\alpha,\alpha[$, which implies that
$(\frac{\mu}{\lambda})^nf_{P^{-1}}(x) \ \in \ ]-\alpha,\alpha[$; it follows that,
$\forall \ x \ \in \ ]\rho-\eta, \rho+\eta[, \ \forall \ n \ \in \ \mathbb{N}$, we have:
\\ $\mid x_n(x)-\rho\mid=\mid f_P((\frac{\mu}{\lambda})^nf_{P^{-1}}(x))-f_P(0)\mid\leq\epsilon$.
Therefore the point $\rho$ is stable.
\\
\item[(2)-(ii)] If $\triangle_A=0$, then the characteristic equation have a double root $\lambda$, given by:
$ \lambda=\frac{tr(A)}{2}=\frac{a+b}{2}$, in this case the matrix $A$ can be expressed as follows: $A=PTP^{-1}$, where
 $P=\left(\begin{array}{cc}
p_{11} & p_{12}\\
p_{21} & p_{22} \\
\end{array} \right)$ is a invertible matrix and
$T=\left(\begin{array}{cc}
\lambda & 1 \\
0 & \lambda \\
\end{array} \right)$.
The vector
$\left(\begin{array}{c}
p_{11} \\
p_{21} \\\end{array}\right)$ is a proper vector associated to the eigenvalue $\lambda$, then the coefficients $p_{11}$
and $p_{21}$ satisfy: $cp_{12}+(d-\lambda)p_{21}=0$. Hence $p_{21}\neq0$. \\
\noindent Now let $(x_n(u))_{n\geq0}$ be the solution of Eq.(A.1) with initial value $u \ \in \ \mathbb{R}$.
We distinguish two cases:
\medskip
\begin{itemize}[leftmargin=*]
\item $u=\frac{\lambda-d}{c}=\frac{a-d}{2c}$. Then
\\
$\displaystyle{
\begin{array}{ccc}
               x_1\hspace{-2mm}&=&\hspace{-2mm}f_A(u) =\displaystyle{\frac{a\frac{a-d}{2c}+b}{c\frac{a-d}{2c}+d}}
=\displaystyle{\frac{a^2-ad+2bc}{c(a+d)}}=\displaystyle{\frac{a^2+ad-2(ad-bc)}{c(a+d)}} \\
                &=&\hspace{-1cm}
\displaystyle{\frac{a^2+ad-\frac{(a+d)^2}{2}}{c(a+d)}}=\displaystyle{\frac{a^2-d^2}{2c(a+d)}}
=\displaystyle{\frac{a-d}{2c}}=u \ \ \ \ \ \ \ \ \ \ \ \ \ \ \
\end{array}} $\\
So $x_n=u, \ \forall \ n \ \in \ \mathbb{N}.$
\medskip
\item $u\neq\frac{\lambda-d}{c}=\frac{p_{11}}{p_{21}}$. Since
$T^n=\left(\begin{array}{cc}
\lambda^n & n\lambda^{n-1} \\
0 & \lambda^n \\
\end{array}\right)=\lambda^{n-1}\left(\begin{array}{cc}
\lambda & n\lambda \\
0 & \lambda \\
\end{array}\right)$, so
$$x_n(u)=f_P(f_{T^n}(f_{P^{-1}}(u)))=f_P(f_{P^{-1}}(u)+\frac{n}{\lambda}) \ _ { \overrightarrow{n\rightarrow+\infty}} \
\frac{p_{11}}{p_{21}}=\frac{a-d}{2c}$$
\end{itemize}
\noindent
Therefore the equilibrium point $\rho=\frac{a-d}{2c}$, is a global attractor of Eq.(A.1),
relative to the set $\mathbb{R}\setminus\mathcal{F}.$\\
\noindent
Now suppose that $\rho$ is stable, and let $\alpha \ \in \mathbb{R}$ such that $f_P(\alpha)\neq\rho$,
then there exists $\eta>0$, satisfied: $\forall \ u \ \in \ \mathbb{R}\setminus\mathcal{F} \ \text{such that}
\ \mid u-\rho\mid\leq\eta, \text{and} \ \forall \ n \in \ \mathbb{N}$,
we have: $\mid x_n(u)-\rho\mid\leq\displaystyle{\frac{\mid f_P(\alpha)-\rho\mid}{2}}.$
Let $u_n=f_P(-\frac{n}{\lambda}+\alpha)$. Then  $\displaystyle{\lim_{n\rightarrow+\infty}u_n=\rho}$, particularly there exists $N\in\mathbb{N}$ such that $\mid u_N-\rho\mid\leq\eta$, which implies that
$\mid x_N(u_N)-\rho\mid=\mid f_P(\alpha)-\rho\mid\leq\displaystyle{\frac{\mid f_P(\alpha)-\rho\mid}{2}}$, so $f_P(\alpha)-\rho=0$, a contradiction. Therefore the point $\rho$ is unstable.\\

\noindent
If $\Delta_A<0$, let $re^{\pm i\theta}$ be the roots of $P_A$. Then by Lemma A.2, the matrix $A^n$
can be expressed as follows:
$$A^{n}=\frac{r^{n-1}\sin(n\theta)}{\sin\theta}\left[A-\frac{r\sin((n-1)\theta)}{\sin (n\theta)}I\right], \ n=1, 2,\cdots$$
which implies that: $\forall \ u \in \ \mathbb{R}\setminus\mathcal{F}$ the solution of Eq.(A.1) with initial value $u$
is given by:
 $$x_{n}(u)=\frac{\left(a-\frac{r\sin((n-1)\theta)}{\sin(n\theta)}\right)u+b}{cu+d-
 \frac{r\sin((n-1)\theta)}{\sin(n\theta)}}=\frac{-r(\cos \theta-\sin \theta\cot(n\theta))u+au+b}{-r(\cos \theta-\sin
 \theta\cot(n\theta))+cu+d}, \ n\geq1$$
\noindent
Then we distinguish two cases:
\\
\noindent
\item[(2)-(iii)] If $\theta=\frac{p}{q}\pi$ with $p \ \in \ \mathbb{Z}, \ q \ \in \mathbb{N}\setminus\{0,1,2\}$ and \ $gcd(p,q)=1$, then all solutions of Eq.(A.1) are $q$-periodic.
\item[(2)-(iv)] If $\theta \ \not\in \ \pi\mathbb{Q}$, then the set
$\{-r(\cos \theta-\sin \theta\cot(n\theta)), \ n\in \mathbb{N}^{*}\}$ is dense in $\mathbb{R}$, so
    $$\overline{\{x_n, \ n \in \mathbb{N}\}}=\overline{\left\{\frac{yu+au+b}{y+cu+d}, \ y\neq-cu-d\right\}}=\mathbb{R},$$
    therefore the solution $(x_n(u))_{n\geq0}$ is dense in $\mathbb{R}$.\\
   \flushright{$\square$}
\end{description}


\medskip

\begin{thebibliography}{9}
\bibitem{2}Z. AlSharawi, J. Angelos, S. Elaydi, L. Rakesh.
\textit{An extension of Sharkovsky's theorem to periodic difference equations}, J. Math. Anal. Appl. \textbf{316} (2006), 128-141.
\bibitem{9}F. Balibrea, M.V. Caballero, \textit{Stability of orbits via lyapunov exponents in autonomous and nonautonomous systems}, International Journal of Bifurcation and Chaos, \textbf{23}, No. 07, (2013).
\bibitem{8} M.E. Clark, L.J. Gross, \textit{Periodic solutions to nonautonomous difference equations}. Maths. Biosciences, \textbf{102} (1990), 105-119.
\bibitem{10} S. Elaydi, R.J. Sacker, \textit{Periodic difference equations, population biology and the Cushing-Henson conjectures}, Mathematical Biosciences,
\textbf{201}, Issues 1–2, (2006), 195-207.
\bibitem{1} E.A. Grove, Y. Kostrov, G. Ladas, SW. Schultz, \textit{Riccati difference equations with real period-2 coefficients},
Commun. Appl. Nonlinear Anal, \textbf{14}, 33-56.
\bibitem{5} V.L. Kocic, G. Ladas, Global Bevavior of Nonlinear Difference Equations of Higher Order with Applications. Kluwer Academic Publishers. Dordrecht, (1993).
\bibitem{6} M.R.S. Kulenovic, G. Ladas, Dynamics of Second order Rational Difference Equations. Chapman and Hall/CRC, Boca Raton, (2002).
\bibitem{11} U. Krause, \textit{Stability of non-autonomous population models with bounded and periodic enforcement}, Journal of Difference Equations and Applications
\textbf{15}, Issue 7, (2009).
\bibitem{7} S. Stevi\'c, \textit{Domains of undefinable solutions of some equations and systems of difference equations}, Appl. Math. Comput. \textbf{219} (2013), 11206-11213.
\end{thebibliography}
\end{document}